\documentclass[12pt,reqno]{amsart}

\usepackage{a4wide}
\usepackage{amsmath}
\usepackage{psfrag}
\usepackage{amsthm}
\usepackage{amssymb,amsxtra}
\usepackage{appendix}
\usepackage[usenames]{color}
\usepackage{stmaryrd}
\usepackage{mathrsfs}
\usepackage{upgreek}
\usepackage{ytableau}

\usepackage{comment}
\usepackage{tikz-cd}
\usetikzlibrary{cd}

\usepackage{graphicx}
\usepackage[all]{xy}
\usepackage{tikz}
\usetikzlibrary{matrix,fit}
\usepackage{geometry}
\usetikzlibrary{positioning}
\usepackage{hyperref}
\usepackage{cleveref}
\usepackage{enumerate}
\usepackage{tabularray}

 \allowdisplaybreaks

\newcommand{\Z}{\mathbb{Z}}

\DeclareMathOperator{\coredim}{\mathsf{c}}

\newcommand{\sym}[1]{\mathfrak{S}_{#1}}

\renewcommand{\P}{\Lambda^+}
\newcommand{\pReg}{\Lambda^+_p}

\newcommand{\R}{\mathbb{R}}

\newcommand{\sgn}{\operatorname{\mathrm{sgn}}}
\newcommand{\C}{\mathbb{C}}

\newcommand{\F}{{F}}
\newcommand{\K}{K}

\newcommand{\Sl}{a}
\newcommand{\Sr}{b}
\newcommand{\Dl}{\ell}
\newcommand{\Dr}{r}
\newcommand{\block}{b}

\renewcommand{\i}{\mathring{\imath}}
\renewcommand{\j}{\mathring{\jmath}}

\renewcommand{\t}{\mathfrak{t}}

\newcommand{\Ext}{\mathsf{Ext}}

\newcommand{\Span}{\mathsf{span}}

\newcommand{\Rect}{\mathsf{R}}

\renewcommand{\mod}{\mathrm{mod}\ }

\newcommand{\Rad}{\mathsf{Rad}}
\newcommand{\Proj}{\mathsf{P}}
\newcommand{\Inj}{\mathsf{I}}

\newcommand{\Res}{\mathrm{Res}}
\newcommand{\Ind}{\mathrm{Ind}}

\newcommand{\rad}{{\mathsf{Rad}}}

\newcommand{\Hom}{{\mathsf{Hom}}}

\renewcommand{\le}{\leqslant}

\renewcommand{\ge}{\geqslant}
\renewcommand{\geq}{\geqslant}

\theoremstyle{definition}

\newtheorem{defn}{Definition}[section]

\newtheorem{eg}[defn]{Example}
\newtheorem{conj}[defn]{Conjecture}

\theoremstyle{plain}

\newtheorem{prop}[defn]{Proposition}
\newtheorem{lem}[defn]{Lemma}
\newtheorem{cor}[defn]{Corollary}
\newtheorem{thm}[defn]{Theorem}

\numberwithin{equation}{section}

\allowdisplaybreaks

\begin{document}
\title[Tensor Product and the Stable Green Ring of $\F\sym{p}$]{Tensor Products and the Stable Green Ring of the Symmetric Group Algebra $\F\sym{p}$}

\author{Manzu Kua}
\address[M. Kua]{Division of Mathematical Sciences, Nanyang Technological University, SPMS-04-01, 21 Nanyang Link, Singapore 637371.}
\email{s220025@e.ntu.edu.sg}

\author{Kay Jin Lim}
\address[K. J. Lim]{Division of Mathematical Sciences, Nanyang Technological University, SPMS-04-01, 21 Nanyang Link, Singapore 637371.}
\email{limkj@ntu.edu.sg}

\begin{abstract}
We give an explicit formula for the decomposition of the tensor product of any two indecomposable non-projective modules for the symmetric group algebra $\F\sym{p}$ modulo projective modules. In particular, we show that the tensor product of two simple modules is semisimple modulo projectives. We also compute the Benson--Symonds invariants for all such indecomposable non-projective modules.
\end{abstract}


\thanks{The authors thank Karin Erdmann for pointing out Proposition \ref{P:HypoA}(iii) and suggesting the computation in Section~\ref{S:BensonSymonds}. They also thank the anonymous referee for a suggestion that led to the addition of Proposition 5.2(vi). The second author is supported by Singapore Ministry of Education AcRF Tier 2 grant MOE-T2EP20225-0003.}

\maketitle

\section{Introduction}

Computing the tensor product of modules over a Hopf algebra is known to be very difficult. Even in the case of group algebras, there is no general method to decompose a tensor product as a direct sum of indecomposable modules. In the semisimple case, this problem is equivalent to multiplying irreducible characters. When we move to the modular case, the problem becomes much harder, since the group algebra often has infinitely many indecomposable modules.

For abelian groups in the ordinary case, the simple modules are one-dimensional, and in principle the tensor product can be read from the character table. For a cyclic group of order $p^s$ over a field of characteristic $p$, the indecomposable modules are precisely the modules $F[x]/(x^i)$ for $1 \le i \le p^s$, which correspond to Jordan blocks of size at most $p^s$ (see \cite{Higman:1954}). The decomposition of the tensor product of two such modules was described by Srinivasan \cite{Srinivasan:1964}. Subsequently, Dade \cite{Dade:1966} classified the indecomposable modules for groups with a normal cyclic Sylow $p$-subgroup. The tensor products of these indecomposable modules were later studied by Feit \cite{Feit:1966}, when the Sylow $p$-subgroup has order $p$, and by Lindsey \cite{Lindsey:1974} in the general case. Janusz \cite{Janusz:1966} also constructed indecomposable modules for groups with cyclic Sylow $p$-subgroups. As a consequence, when $G$ has a cyclic Sylow $p$-subgroup $P$ of order $p$, the tensor product of two indecomposable $\F G$-modules is isomorphic, modulo projective $\F G$-modules, to the induction of the tensor product of their Green correspondents.

For the symmetric group, in the ordinary case the irreducible modules are the Specht modules. Computing the coefficient of a Specht module in the decomposition of the tensor product of two Specht modules (these coefficients are known as the Kronecker coefficients) is known to be \#P-hard. In the theory of symmetric functions, this tensor product corresponds to the internal product in the ring of symmetric functions.

Consider now the modular case. Related to the decomposition problem is the block decomposition of the group algebra. The blocks of the symmetric groups have been classified by Brauer and Robinson (see \cite{Brauer:1947,robinsonproof}),  confirming the Nakayama conjecture, which states that two Specht modules lie in the same block if and only if their corresponding partitions have the same $p$-core. Furthermore, in the modular case, the Specht modules are no longer simple. Instead, the simple modules are parametrised by the $p$-regular partitions. They appear as the heads of certain Specht modules, as constructed by James \cite{james}. The decomposition of the tensor product of two Specht modules, or of two simple modules, remains poorly understood. It is known only in a few special cases, for example when one of the modules is the signature representation of the symmetric group. In particular, the tensor product of a Specht module with the signature representation is isomorphic to the dual of the Specht module labelled by the conjugate partition of the original one. On the other hand, the tensor product of a simple module labelled by $\lambda$ with the signature representation is the simple module labelled by another partition determined by the Mullineux map (see \cite{Ford/Kleshchev:1997,Mullineux:1979}).

In this paper, we study the tensor products of indecomposable modules for the group algebra $\F\sym{p}$. As mentioned earlier, since $\sym{p}$ has a cyclic Sylow $p$-subgroup of order $p$, this problem could in principle be approached using the general theory described above. However, obtaining an explicit decomposition still requires further work. Instead, we make use of the Ext-quiver of the principal block $\block_0$ of $\F\sym{p}$, together with the composition factors of the Specht modules and the Littlewood--Richardson rule. We show that the indecomposable non-projective $\F\sym{p}$-modules are precisely the syzygies of the simple modules in $\block_0$, and we obtain an explicit formula for the decomposition of their tensor products modulo projectives. In particular, this shows that the tensor product of two simple modules is semisimple modulo projectives. In the final section of the paper, we compute the Benson--Symonds invariant for all indecomposable $\F\sym{p}$-modules.

\section{Preliminaries}  For basic knowledge of the representation theory of finite-dimensional algebras, we refer the reader to \cite{alp,ben1}. For the representation theory of symmetric groups, we refer the reader to \cite{james,jk}.

Throughout, for any $a,b\in\R$, let $[a,b]$ be the set consisting of all integers in between $a$ and $b$ inclusively. We also let $\F$ be an algebraically closed field of characteristic $p > 0$. 

Let $A$ be a finite-dimensional $\F$-algebra and $\Rad(M)$ be the (Jacobson) radical of an $A$-module $M$. In general, we have $\Rad^i (M) = \Rad (\Rad^{i-1}(M))$ for $i \geq 2$. By convention, $\Rad^0(M)=M$. The quotient $\Rad^{i-1} (M) / \Rad^i (M)$ is called the $i$-th radical layer of $M$. We write $\Proj(M)$ for the projective cover of $M$. If $M_1, M_2, \ldots , M_r$ is a complete set of composition factors of $M$, then we write 
\[M \sim M_1 + M_2 + \cdots + M_r.\]
If $M$, $M'$ and $M''$ are modules such that the composition factors of $M'$ and $M''$ together form a complete list of composition factors of $M$, then we also write
\[M \sim M' + M''.\]
If $\pi : \Proj(M) \twoheadrightarrow M$ is a surjection onto $M$, we let $\Omega (M) := \ker \pi$. Inductively, one defines $\Omega^i(M) := \Omega (\Omega^{i-1} (M))$ for $i\ge 2$. On the other hand, let $\Omega^{-1} (M)$ be the cokernel of an injection $\iota : M \hookrightarrow \Inj(M)$ where $\Inj(M)$ is the injective hull of $M$. Inductively, one defines $\Omega^{-i}(M) := \Omega^{-1}(\Omega^{-i+1}(M))$ for $i \geq 2$. By convention, $\Omega^0 (M)$ is the largest non-projective summand of $M$, and is referred to as the core of $M$. These are the Heller syzygies of the module $M$ and are well-defined up to isomorphism.  
 
Let $\Ext_A^n(M,M')$ be the $n$-th extension group for any given two $A$-modules $M$ and $M'$. For any two simple $A$-modules $T,T'$, we have \[\dim_\F \Ext^1_A(T,T') = \dim_\F \Hom_A(T' , \rad (\Proj(T))/\rad^2 (\Proj(T))).\] 
If $T_1, T_2, \ldots, T_\ell$ is a complete set of distinct simple $A$-modules, the Ext-quiver of $A$ is the directed graph with vertices $\{1, 2, \ldots , \ell\}$, and the number of arrows $i \to j$ is equal to $\dim_\F (\Ext^1_A (T_i, T_j))$.

\begin{lem}\label{lem:Ext1}\ 
	\begin{enumerate}[(i)]
		\item Let $M$ be an indecomposable $A$-module such that $\Rad^2(M)=0$ and $T$ be a composition factor of $M/\Rad(M)$. If $\Rad(M)\ne 0$, then there exists a composition factor $T'$ of $\Rad(M)$ such that $\Ext^1_A(T,T')\neq 0$. 
		\item Let $M$ be an $A$-module. If $\Ext^1_A(T,T')=0$ for any two composition factors $T,T'$ of $M$, then $M$ is semisimple. 
	\end{enumerate}
\end{lem}
\begin{proof} For part (i), suppose on the contrary that $\Ext^1_A(T,T')=0$ for all such $T'$. We have 
	\[\xymatrix{&\Proj(T)\ar@{->>}[d]^\alpha\ar@{-->}[dl]_f\\ M\ar@{->>}[r]^\pi&T\ar[r]&0}\] In particular, $f:\Rad^i(\Proj(T))\to \Rad^i(M)$ for any $i\ge 1$. The map $f$ induces $\bar f:\Rad(\Proj(T))/\Rad^2(\Proj(T))\to \Rad(M)/\Rad^2(M)=\Rad(M)$ defined by $\bar f(\bar u)=f(u)$ for any $u\in\Rad(\Proj(T))$. Since $\Ext^1_A(T,T')=0$ for any composition factor $T'$ in $\Rad(M)$, we have $\bar f=0$. As such, $f$ annihilates $\Rad(\Proj(T))=\ker\alpha$ and it induces $f':\Proj(T)/\Rad(\Proj(T))\to M$ defined by $f'(\bar y)=f(y)$. Let $\beta:T\to \Proj(T)/\Rad(\Proj(T))$ be $\beta(x)=\bar y$ if $\alpha(y)=x$. Then $f'\circ \beta$ is a section of $\pi$. Since $M$ is indecomposable, we have $M\cong T$. This contradicts the fact that $\Rad(M)\neq 0$. 
	
	For part (ii), let $N$ be a non-simple indecomposable summand of $M$. Consider an indecomposable summand $V$ of $N/\Rad^2(N)$ such that $\Rad(V)\ne 0$. Since $\Rad^2(V)=0$, part (i) applies and there are composition factors $T,T'$ of $V$ and hence of $M$ such that $\Ext^1_A(T,T')\ne 0$, yielding a contradiction. Thus, every indecomposable summand of $M$ is simple. 
\end{proof}

We now consider modules for group algebras. Let $G$ be a finite group. We write $a(\F G)$ for the Green ring of the group algebra $\F G$. The stable Green ring of $\F G$ has a $\Z$-basis the set of isomorphism classes of indecomposable modules in the stable category of $\F G$. Its addition and multiplication are induced by the direct sum and tensor product of modules respectively. 

Let $M$ be an $\F G$-module. For a subgroup $H$ of $G$, we write $\Res_H(M)$ for the restriction of $M$ to $H$. If $M'$ is a $\F H$-module, we write $\Ind_{H}^G (M')$ for the induction of $M'$ to $G$. The dual module of $M$ is denoted as $M^*$. By abuse of notation, the trivial $\F G$-module is also denoted as $\F$. For any positive integer $n$, let $\coredim_n^G(M) := \dim_\F \Omega^0 (M^{\otimes n})$.  The Benson--Symonds invariant  is  the quantity
\[\gamma_G(M) := \limsup_{n \to \infty} \sqrt[n]{\coredim_n^G(M)}.\] We collate together some results we shall need from their paper.

 \begin{thm}[\cite{bensonsymonds}]\label{T:BensonSymonds} Let $M$ be an $\F G$-module. 
 \begin{enumerate}[(i)]
   \item For any short exact sequence of $\F G$-modules $0\to M_1\to M_2\to M_3\to 0$, we have $\gamma_G(M_2)\le \gamma_G(M_1)+\gamma_G(M_3)$. Furthermore, if the sequence splits, then we also have $\max\{\gamma_G(M_1),\gamma_G(M_3)\}\le \gamma_G(M_2)$. In particular, $\gamma_G(M)=\gamma_G(\Omega^0(M))$.
   \item We have $\gamma_G(\Omega(M))=\gamma_G(M)$.
   \item We have $\gamma_G(M)=\max_{E\le G}\gamma_E(M)$ where the maximum is taken over all elementary abelian $p$-subgroups of $G$ up to conjugation. 
   \item We have $\gamma_G(M^*)=\gamma_G(M)$.
 \end{enumerate}
 \end{thm}
 \begin{proof} We only require further justification for part (iii). Let $g\in G$ and $E$ be an elementary abelian $p$-subgroup of $G$. Suppose that $M^{\otimes n}= N\oplus P$ where $N\cong\Omega^0(M^{\otimes n})$ and $\Res_E(N)=N'\oplus Q$ where $N'\cong \Omega^0(\Res_E(N))$. Then $\Res_{{}^gE}(N)\cong gN'\oplus gQ$ where $gN'\cong \Omega^0(\Res_{{}^gE}(N))$. Therefore, $\gamma_E(M)=\gamma_{{}^gE}(M)$.
 \end{proof}

We now focus in particular on the modules for symmetric groups. The symmetric group on $n$ letters is denoted as $\sym{n}$. Let $\P(n)$ and $\pReg(n)$  be the sets of partitions and $p$-regular partitions of $n$ respectively. For a composition $\delta=(\delta_1,\ldots,\delta_\ell)$ of $n$, we write $\sym{\delta}$ for the Young subgroup of $\sym{n}$ which is canonically isomorphic with the direct product $\sym{\delta_1} \times \sym{\delta_2} \times \cdots \times \sym{\delta_\ell}$. Let $\lambda\in\P(n)$. The symmetric group acts naturally on the set of all $\lambda$-tableaux. For a $\lambda$-tableau $T$, let $R_T$ and $C_T$ be its row and column stabilisers respectively. A $\lambda$-tabloid is an equivalence class $\{T\}$ of $\lambda$-tableaux where $T$ and $T'$ represent the same tabloid if and only if $T = \sigma T'$ for some $\sigma \in R_T$. Defining $\kappa_T := \sum_{\sigma \in C_T} (\sgn \sigma) \sigma$, we obtain the $\lambda$-polytabloid $e_T := \kappa_T\{T\}$. 

Let $\lambda \in \P(n)$. The Specht module $S^\lambda$ is the submodule of the permutation module $M^\lambda:=\Ind_{\sym{\lambda}}^{\sym{n}}\F$ spanned by the various polytabloids $e_T$'s. The module $S^\lambda$ has a basis $e_T$'s where $T$ runs through all standard $\lambda$-tableaux. Furthermore, \[(S^{\lambda'})^*\cong S^\lambda\otimes \sgn\] where $\lambda'$ is the conjugate partition of $\lambda$ and $\sgn$ is the signature representation. The simple modules for $\F\sym{n}$ are parametrised by the set of $p$-regular partitions $\pReg(n)$ and we denote them by $D^\lambda$, one for each $\lambda\in\pReg(n)$. Furthermore, we have $D^\lambda\cong S^\lambda/\Rad(S^\lambda)$ and $D^\lambda$ is self-dual. The Specht module $S^\lambda$ is indecomposable when $p$ is odd or $p=2$ and $\lambda$ is $2$-regular. By a result of Brauer and Robinson (see \cite{Brauer:1947,robinsonproof})  which is commonly known as the Nakayama conjecture, the blocks of $\F\sym{n}$ are in bijection with the $p$-cores of the partitions of $n$.

In this paper, we further focus on the group algebra $\F \sym{p}$. In this case, the $p$-core of a partition $\lambda\in\P(p)$ is either the partition itself or the empty partition. In the first case, we have $D^\lambda\cong S^\lambda\cong \Proj(D^\lambda)$. On the other hand, the block labelled by the empty partition is the principal block $\block_0$ of $\F\sym{p}$. In this case, since the $p$-core of $\lambda$ is empty, $\lambda$ must be a hook of size $p$, i.e., $\lambda=(p-i,1^i)$ for some $i\in [0,p-1]$. For $i\in [0,p-2]$, we write $D_i$ in place of $D^\lambda$ and $P_i=\Proj(D_i)$. Similarly, we write $S_i$ for $S^\lambda$ for $i\in [0,p-1]$. Finally, the direct summand of an $\F \sym{p}$-module $M$ belonging in $b_0$ is denoted as $(M)_0$. 

The Specht modules $S_0$ and $S_{p-1}$ are the trivial and sign modules respectively. For $i\in [1,p-2]$, the composition factors of $S_i$ are $D_i$ and $D_{i+1}$ with $D_{i+1}$ the unique maximal submodule of $S_i$. The radical layers of $P_i$ are also well studied and can be depicted as follows:

\begin{align*}
    P_0&={\raisebox{5ex}{\xymatrix@R=2mm@C=1.5mm{D_0\ar@{-}[d]\\ D_1\ar@{-}[d]\\ D_0}}},& 
    P_i&={\raisebox{5ex}{\xymatrix@R=2mm@C=1.5mm{&D_i\ar@{-}[dr]\ar@{-}[dl]&\\ D_{i-1}\ar@{-}[dr]&&D_{i+1}\ar@{-}[dl]\\ &D_i&}}},& 
    P_{p-2}&={\raisebox{5ex}{\xymatrix@R=2mm@C=1.5mm{D_{p-2}\ar@{-}[d]\\ D_{p-3}\ar@{-}[d]\\ D_{p-2}}}}.
\end{align*} In particular, we obtain the Ext-quiver of $\block_0$ as follows where the vertex $i$ labels the simple module $D_i$:
\[
\begin{tikzcd}[every arrow/.append style={shift left, -stealth}]
	\ \bullet\ar[phantom, "0"above=4pt] \arrow[r]   & \arrow[l]  \ \bullet\ar[phantom, "1"above=4pt] \arrow[r] & \arrow[l] \ \bullet\ar[phantom, "2"above=4pt] \arrow[r] & \arrow[l] \bullet\  \cdots \ \bullet  \arrow[r] & \arrow[l] \bullet\ar[phantom, "p-2"above=4pt]
\end{tikzcd}
\]

It is a well-known fact that the principal block of $\F\sym{p}$ is a Brauer tree algebra. As Janusz \cite[6.1 Theorem]{Janusz:1966} proved, the number of indecomposable modules of a Brauer tree algebra is given by $e(em+1)$ where $e$ is the number of edges in the Brauer tree and $m$ is the multiplicity of its exceptional vertex. In our case, there is no exceptional vertex and $e=p-1$ for $\block_0$.  

\begin{thm}\label{T:NumInd} The number of indecomposable modules in $\block_0$ is $p(p-1)$. 
\end{thm}



We end this section with a result regarding the Heller translates of the simple modules $D_j$. 

\begin{lem}\label{lem: limtan} \cite[Lemma 2.5]{limtan} Let $p \geq 3$ and $i\in\Z$. Suppose that $i \equiv k \pmod{2p-2}$ with $k\in [0,2p-3]$. Then 
	\[\Omega^i (D_0) \cong \left \{\begin{array}{ll}
		S_k&\text{if $k\in [0,p-1]$,} \\
		S_{2p-k-2}^*& \text{if $k\in [p,2p-3]$.}
    \end{array}\right .\]
\end{lem}

\begin{cor}\label{cor: SiDj is Heller} Let $p\ge 3$, $i\in \Z$ and $j\in [0,p-2]$. Suppose that $i \equiv k \pmod{2p-2}$ with $k\in [0,2p-3]$. Then \[\Omega^i(D_j)\cong \left \{\begin{array}{ll}\Omega^0(S_k\otimes D_j)&\text{if $k\in [0,p-1]$,}\\ \Omega^0(S_{2p-k-2}^*\otimes D_j)&\text{if $k\in [p,2p-3]$.}\end{array}\right .\] In particular, the non-projective part of $S_k\otimes D_j$ (respectively, $S^*_{2p-k-2}\otimes D_j$) is nonzero indecomposable. 
\end{cor}
\begin{proof} Use Lemma \ref{lem: limtan} and the fact that $\Omega^i(U\otimes V)\cong \Omega^0(\Omega^i(U)\otimes V)$.
\end{proof}

\section{Heller Translates of the Simple Modules}

In this section, we examine the Heller translates of the simple modules in the principal block $\block_0$ of $\F\sym{p}$. Our result expresses the Heller translates of a simple module in terms of the tensor products of the simple module with certain Specht or dual Specht modules (see Corollary \ref{cor:0block&Heller}).

Our method begins with studying the composition factors of $P_i\otimes D_j$ belonging in the principal block $\block_0$ and then obtaining such information for $S_i\otimes D_j$. To begin, we require the following two lemmas which study the restriction of some Specht and simple modules to certain Young subgroup.

\begin{lem}\label{lem: resSpecht} Let $i\in [0,p-2]$ and $H=\sym{p-i-1}\times \sym{i+1}$. Then \[ \Res_H S_1 \cong \left \{\begin{array}{ll}(S^{(p-2,1)}\boxtimes \F)\oplus (\F \boxtimes S^{(1)})&\text{if $i=0$,}\\ 
(S^{(1)} \boxtimes \F) \oplus (\F \boxtimes S^{(p-2,1)}) &\text{if $i=p-2$,}\\ 
(S^{(p-i-2,1)} \boxtimes \F) \oplus (\F \boxtimes S^{(i,1)}) \oplus (\F \boxtimes \F)&\text{if $i\neq 0,p-2$.}\end{array}\right .\]
\end{lem}
\begin{proof} The restriction to $H$ is semisimple. When $i=0$, the isomorphism follows from the branching rule. When $i=p-2$, the subgroup is conjugate to $\sym{p-1}\times \sym{1}$ and hence it follows from the previous case. We now assume that $i\neq 0,p-2$. For any $k\in [1,p]$, let $\t_k$ be the $(p-1,1)$-tabloid with $k$ in its second row and let $e_k=\t_k-\t_1$ for $k\neq 1$. The Specht module $S_1$ has a basis $\{e_k:k\in[2,p]\}$. The Young subgroup $H$ is generated by $\alpha = (1, 2, \ldots , p-i-1)$, $\beta = (1,2)$, $\sigma = (p-i, p-i+1, \ldots , p)$, and $\tau = (p-i, p-i+1)$. 

It is easy to check that 
\begin{align*}
\alpha e_k &= 
\left \{\begin{array}{ll}
	e_{k+1}-e_2 &\text{if $k\in[2,p-i-2]$,} \\
	-e_2 &\text{if $k = p-i-1$,} 
\end{array}\right .&
\beta e_k &=
\left \{\begin{array}{ll}
	-e_2 &\text{if $k=2$,} \\
	e_k-e_2 &\text{if $k\in [3,p-i-1]$.} 
\end{array}\right .
\end{align*} Furthermore, both $\sigma,\tau$ fix $e_k$ for $k\in [2,p-i-1]$. Thus, $U:=\Span \{e_2, \ldots , e_{p-i-1}\}\cong S^{(p-i-2,1)}\boxtimes \F$ is an indecomposable summand of $\Res_HS_1$. 

We now consider $W:=(\Res_HS_1)/U$. It has a basis $\{\overline{e_{p-i}}, \ldots, \overline{e_p}\}$. We claim that $W\cong \F \boxtimes M^{(i,1)}$. Again, we leave it to the reader to check that both $\alpha$ and $\beta$ fix $\overline{e_k}$ for $k\in [p-i,p]$. Since the second factor $L$ of $H$ acts only on the numbers $p-i,\ldots,p$, for any $\delta\in L$, we have \[\delta\overline{e_k}=\overline{\t_{\delta(k)}-\t_{\delta(2)}}= \overline{\t_{\delta(k)}-\t_2}=\overline{e_{\delta(k)}}.\] The stabiliser of $\overline{e_p}$ is the permutation subgroup $L'\cong \sym{i}$ of $L$ fixing the number $p$. As such, as $\F L$-module, $W$ is isomorphic to $\Ind^L_{L'} \F$. This proves the claim. Finally, since $M^{(i,1)}\cong S^{(i,1)}\oplus \F$, our proof is complete. 
\end{proof}

\begin{lem}\label{lem: resSimple} Let $i,j\in [0,p-2]$ and $H=\sym{p-i-1}\times \sym{i+1}$. Then 
\[\Res_H D_j \cong \bigoplus \big( S^{(p-i-1-r,1^r)} \boxtimes S^{(i-j+r+1, 1^{j-r})} \big)\] where $r$ runs through the integer interval $[0,p-i-2]\cap [j-i,j]$. 
\end{lem}
\begin{proof} Consider the short exact sequence \[ 0 \to \F \boxtimes \F \to \Res_H S_1 \to \Res_H D_1 \to 0.\] By Lemma \ref{lem: resSpecht}, since we are in the semisimple case, we have 
\[ \Res_H D_1 \cong \left \{\begin{array}{ll}S^{(p-2,1)}\boxtimes \F&\text{if $i=0$,}\\ 
\F \boxtimes S^{(p-2,1)} &\text{if $i=p-2$,}\\ 
(S^{(p-i-2,1)} \boxtimes \F) \oplus (\F \boxtimes S^{(i,1)})&\text{if $i\neq 0,p-2$.}\end{array}\right .\] Let $V=S^{(p-i-2,1)}\boxtimes \F$ and $W=\F \boxtimes S^{(i,1)}$. For any $r\ge 0$, we have \[\bigwedge^rV\cong (\bigwedge^rS^{(p-i-2,1)})\boxtimes \F \cong \left \{ \begin{array}{ll}
 S^{(p-i-1-r,1^r)}\boxtimes \F&\text{if $r\in [0,p-i-2]$,}\\ 
 0&\text{otherwise.}\end{array}\right .\] Similarly, $\bigwedge^{j-r}W\cong \F \boxtimes S^{(i-j+r+1,1^{j-r})}$ if $r\in [j-i,j]$ and 0 otherwise. Since $D_j\cong \bigwedge^j D_1$ (that includes $j=0$ and can be found for example in \cite[\S2.2]{limwang}) and using $\bigwedge^j (V\oplus W)\cong \bigoplus_{r=0}^j (\bigwedge^rV)\otimes (\bigwedge^{j-r}W)$ (see, for instance, \cite[Lemma 1.14.3]{bensonpgroups}), we have the desired isomorphism in our statement. 
\end{proof}

We can now describe the composition factors of $(P_i\otimes D_j)_0$ in terms of the Specht modules. For this, we make use of the fact that $P_i$ is a signed permutation module. 

\begin{lem}\label{lem: PitimesDj} Let $i,j\in [0,p-2]$. We have
	\[(P_i \otimes D_j)_0 \sim \sum \big( S_{i-j+2r} + S_{i-j+2r+1}\big)\] where $r$ runs through $[0,p-i-2]\cap [j-i,j]$.
\end{lem}
\begin{proof} Let $H=\sym{p-i-1}\times \sym{i+1}$. By \cite[Theorem 1.4]{glow}, we have \[P_i\cong M((p-i-1)|(i+1))=\Ind^{\sym{p}}_H(\F \boxtimes \sgn).\] Using Lemma \ref{lem: resSimple}, we have \begin{align*}
  P_i\otimes D_j&\cong \big (\Ind^{\sym{p}}_H(\F \boxtimes \sgn)\big )\otimes D_j\\
  &\cong \Ind^{\sym{p}}_H\big ((\F \boxtimes \sgn)\otimes \bigoplus \big ( S^{(p-i-1-r,1^r)} \boxtimes S^{(i-j+r+1, 1^{j-r})} \big)\big )\\
  &\cong \bigoplus \Ind^{\sym{p}}_H \big( S^{(p-i-1-r,1^r)} \boxtimes \big (S^{(i-j+r+1, 1^{j-r})}\otimes \sgn \big)\big )
\end{align*} where $r\in [0,p-i-2]\cap [j-i,j]$. Since $S^{(j-r+1, 1^{i-j+r})}$ is simple, we have  \[S^{(i-j+r+1, 1^{j-r})}\otimes \sgn\cong (S^{(j-r+1, 1^{i-j+r})})^*\cong S^{(j-r+1, 1^{i-j+r})}.\] So $(P_i\otimes D_j)_0\sim \sum \Ind^{\sym{p}}_H \big( S^{(p-i-1-r,1^r)} \boxtimes S^{(j-r+1, 1^{i-j+r})}\big )_0$.  By the Littlewood--Richardson Rule, we have \[\Ind^{\sym{p}}_H \big( S^{(p-i-1-r,1^r)} \boxtimes S^{(j-r+1, 1^{i-j+r})}\big )_0\sim S_{i-j+2r} + S_{i-j+2r+1}.\] This completes the proof. 
\end{proof}

Our next task is to describe the composition factors of $(S_i\otimes D_j)_0$. For a cleaner description, we introduce the following combinatorial gadget.  

\begin{defn}\label{D:Rect} Consider the grid of squares with the horizontal (respectively, vertical) lines labelled by $0,1,\ldots,p$ from top to bottom (respectively, $0,1,\ldots,p-2$ from left to right). Each grid point is labelled by $(a,b)$ if it belongs in the horizontal and vertical lines labelled by $a$ and $b$ respectively. For $j \in [0, p-2]$, consider the rectangle (in red in Figure \ref{Fig:diagram}) with vertices the grid points $(0,j)$, $(j,0)$, $(p-2,p-2-j)$ and $(p-2-j,p-2)$. We call it the $j$-rectangle. We shall see that the horizontal line labelled by $i$ correspond to the $i$-th Heller translates of $D_j$. As such, we also call the line the $i$-th layer. The labels of the vertical lines correspond to the simple modules $D_0,\ldots,D_{p-2}$. For each $i$-th layer, we let
\begin{align*}
	\Sl_{i,j} &= \left \{\begin{array}{ll}
		j-i &\text{if $i \in [0,j]$,} \\ 
		i-j-1&\text{if $i \in [j+1, p-1]$,} 
	\end{array}\right . \\
	\Sr_{i,j} &= \left \{\begin{array}{ll}
		j+i &\text{if $i \in [0, p-2-j]$,}\\
		2p-j-3-i&\text{if $i \in [p-j-1, p-1]$.}
	\end{array}\right .
\end{align*}  The grid points $(i,\Sl_{i,j})$ and $(i,\Sr_{i,j})$ in the $i$-th layer are indicated as blue points in Figure \ref{Fig:diagram}. On the other hand, for each $i\in [0,p-2]$, let $(i,\Dl_{i,j})$ and $(i,\Dr_{i,j})$ be the leftmost and rightmost grid points in the $i$-th layer of the $j$-rectangle respectively, that is, $\Dl_{i,j}=\Sl_{i,j}+\delta_{i>j}$ and $\Dr_{i,j}=\Sr_{i,j}-\delta_{i>p-2-j}$. In total, such a decoration is called the $j$-diagram. Finally, we let \[\Rect(i,j)=\{\Dl_{i,j}+2k:k\in [0,(\Dr_{i,j}-\Dl_{i,j})/2]\},\] that is, $\Rect(i,j)$ labels every other grid point in the $i$-th layer of the $j$-rectangle from $\Dl_{i,j}$ to $\Dr_{i,j}$.

\begin{figure}[htbp]
\[\scalebox{0.5}{
\begin{tikzpicture}
\foreach \i in {2,3}
{\draw[dashed] (\i,8) -- (\i,-1);
\draw[dashed] (\i+3,8) -- (\i+3,3.5);
\draw[dashed] (\i+5,8) -- (\i+5,3.5);
\draw[dashed] (\i+11,-8) -- (\i+11,1);
\draw[dashed] (\i+6,-8) -- (\i+6,-3.5);}

\draw[dashed] (10,-8) -- (10,-3.5);
\draw[dashed] (12,-8) -- (12,1);

\foreach \i in {0,1,2}
{\draw[dashed] (9,\i+5) -- (0,\i+5);
\draw[dashed] (4.5,\i) -- (0,\i);
\draw[dashed] (11.5,\i-2) -- (16,\i-2);
\draw[dashed] (7,\i-7) -- (16,\i-7);}

\draw[dashed] (0,3) -- (4.5,3);
\draw[dashed] (7,-4) -- (16,-4);

\foreach \i in {2,3,5,6,7}
{\draw[thick] (\i,\i)--(\i+7,\i-7);}

\foreach \j in {0,1,2,6,7}
{\draw[thick] (\j+2,-\j+2)--(\j+7,-\j+7);}

\draw[very thick,color=red] (7,7)--(2,2)--(9,-5)--(14,0)--cycle;

\foreach \k in {0,1,2.5,4,5}
{\node at (\k+5,\k-2) {\rotatebox{-45}{\Large$\cdots$}};}

\foreach \l in {0,1,6,7}
{\node at (\l+4.5,-\l+3.5) {\rotatebox{45}{\Large$\cdots$}};}


\foreach \i/\label in {1/0,2/1,4/j-2,5/j-1,6/j,7/j+1}
{\node at (\i+1,9) {\rotatebox{90}{$\label$}};}

\foreach \i/\label in {8/p-3-j,9/p-2-j,10/p-1-j,12/p-4,13/p-3,14/p-2}
{\node at (\i,-9) {\rotatebox{90}{$\label$}};}

\foreach \k/\label in {7/0, 6/1, 5/2, 3/j-1, 2/j, 1/j+1, 0/j+2}
{\node at (-1,\k) {$\label$};}

\foreach \k/\label in {-7/p, -6/p-1, -5/p-2, -4/p-3, -2/p-j, -1/p-1-j, 0/p-2-j}
{\node at (17,\k) {$\label$};}

\foreach \i/\j in {7/7, 6/6, 8/6, 5/5, 9/5, 3/3, 2/2, 2/1, 3/0, 8/-5, 9/-6, 10/-5, 13/-2, 14/-1, 14/0, 13/1}
{\node at (\i,\j) {\textcolor{blue}{\huge{$\bullet$}}};}

\node at (-1,4) {$\vdots$};
\node at (17,-3) {$\vdots$};
\node at (4,9) {$\cdots$};
\node at (11,-9) {$\cdots$};
\end{tikzpicture}}
\]
\caption{$j$-diagram}
\label{Fig:diagram}
\end{figure}
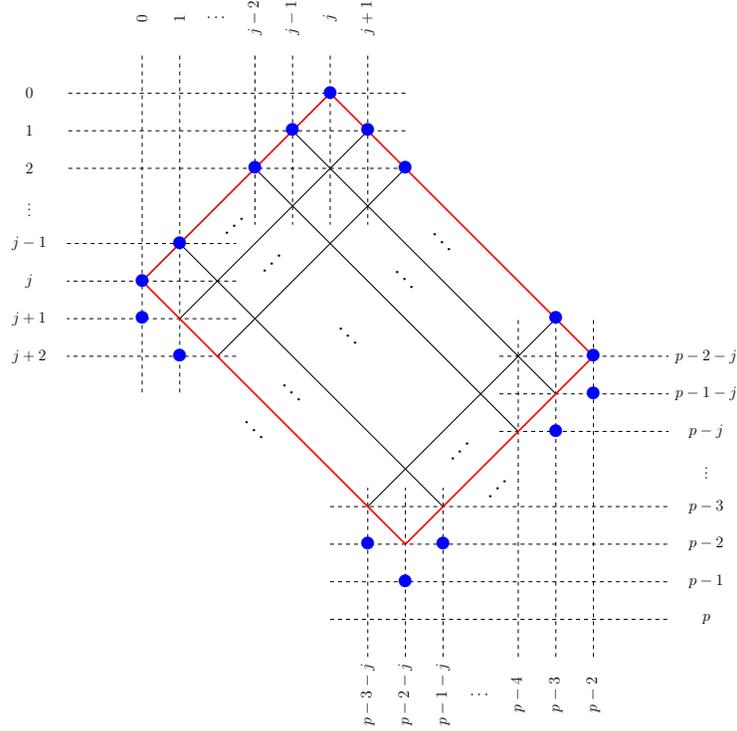
\end{defn}

\begin{lem}\label{lem:compSiDj} Let $i\in [0, p-1]$ and $j\in [0,p-2]$. Then \[(S_i\otimes D_j)_0\sim \sum_{t=\Sl_{i,j}}^{\Sr_{i,j}} D_t.\]
\end{lem}
\begin{proof} Fix $j$. We argue by induction on $i$. For $i=0$, we have $\Sl_{0,j}=\Sr_{0,j}=j$ and, on the other hand, \[(S_0\otimes D_j)_0=D_j.\] Suppose now that the formula holds for some $i\in [0,p-2]$. To complete our proof, since $P_i\sim S_i+S_{i+1}$, by Lemma \ref{lem: PitimesDj}, we just need to check that \begin{equation}\label{Eq:DSpecht}
  \sum_{t=\Sl_{i,j}}^{\Sr_{i,j}} D_t+\sum_{t=\Sl_{i+1,j}}^{\Sr_{i+1,j}} D_t\sim \sum \big( S_{i-j+2q} + S_{i-j+2q+1}\big)
\end{equation} where $q\in [0,p-i-2]\cap [j-i,j]$. To check this, we make use of the known fact about the decomposition numbers of the principal block of $\F \sym{p}$. We prove the case $j< (p-1)/2$ and leave the other case to the reader. Even in this case, there are a few cases for $i$ and we only demonstrate the cases when $i=j$, $i\in [j+1,p-3-j]$ and $i\in [p-1-j,p-1]$. 

Suppose that $i=j$. We have $\Sl_{i,j}=0$, $\Sr_{i,j}=2j$, $\Sl_{i+1,j}=0$ and $\Sr_{i+1,j}=2j+1$. In this case, $q\in [0,j]$. The right-hand side of Equation \ref{Eq:DSpecht} reads \[\sum_{k=0}^{2j+1} S_k\sim \sum_{k=0}^{2j} D_k+ \sum_{k=0}^{2j+1}D_k.\] 

Suppose that  $i\in [j+1,p-3-j]$. We have $\Sl_{i,j}=i-j-1$, $\Sr_{i,j}=i+j$, $\Sl_{i+1,j}=i-j$ and $\Sr_{i+1,j}=i+j+1$. In this case, $q\in [0,j]$. The right-hand side of Equation \ref{Eq:DSpecht} reads \[\sum_{k=i-j}^{i+j+1} S_k\sim \sum_{k=i-j-1}^{i+j} D_k+ \sum_{k=i-j}^{i+j+1}D_k.\] 

Suppose that $i\in [p-1-j,p-1]$. We have $\Sl_{i,j}=i-j-1$, $\Sr_{i,j}=2p-j-3-i$, $\Sl_{i+1,j}=i-j$ and $\Sr_{i+1,j}=2p-j-4-i$. In this case, $q\in [0,p-i-2]$. The right-hand side of Equation \ref{Eq:DSpecht} reads \[\sum_{k=i-j}^{2p-i-j-3} S_k\sim \sum_{k=i-j-1}^{2p-i-j-4} D_k+ \sum_{k=i-j}^{2p-i-j-3}D_k\sim \sum_{k=i-j}^{2p-i-j-4} D_k+ \sum_{k=i-j-1}^{2p-i-j-3}D_k.\] 
\end{proof}

The preceding lemma allows us to understand $S_i\otimes D_j$ better. The following corollary shows that the following three objects are isomorphic: the core of $S_i\otimes D_j$, $(S_i\otimes D_j)_0$ and the syzygy $\Omega^i(D_j)$.

\begin{cor}\label{cor:0block&Heller} Let $p\ge 3$ and $0\le i\le 2p-3$. Then \[\Omega^i(D_j)\cong \left \{\begin{array}{ll} (S_i\otimes D_j)_0&\text{if $i\in [0,p-1]$,}\\ (S_{2p-2-i}^*\otimes D_j)_0&\text{if $i\in [p,2p-3]$.}\end{array}\right .\] 
\end{cor}
\begin{proof} Suppose first that $0\le i\le p-1$. Lemma \ref{lem:compSiDj} asserts that $V:=(S_i\otimes D_j)_0$ is multiplicity-free. Since this is not the case for the projective modules in the principal block of $\F\sym{p}$, we have $\Omega^0(V)\cong V$. Therefore, $V\cong \Omega^i(D_j)$ using Corollary \ref{cor: SiDj is Heller}. For $i\ge p$, we have $(S_{2p-2-i}^*\otimes D_j)_0\sim (S_{2p-2-i}\otimes D_j)_0$. Again, $V':=(S_{2p-2-i}^*\otimes D_j)_0$ is multiplicity-free and hence it has no projective summand. Thus, $V'\cong \Omega^0(V')\cong \Omega^i(D_j)$.
\end{proof}

\section{The stable Green ring of $\F\sym{p}$}

In this section, we prove Theorem \ref{T:StableGreenRing}, the main result of our paper. It gives us a closed form for the tensor product of any two indecomposable non-projective $\F\sym{p}$-modules. As such, we obtain the structure for the stable Green ring for $\F\sym{p}$. As a corollary, we get the classical results regarding the radical structures of the syzygies of the simple modules. 

We first begin with the tensor product of two simple modules. Recall the notation $\Rect(i,j)$ from Definition \ref{D:Rect}.

\begin{lem}\label{L: DiDj semisimple} Let $i,j\in [0, p-2]$. Then \[\Omega^0(D_i\otimes D_j)\cong (D_i\otimes D_j)_0\cong \bigoplus_{t\in\Rect(i,j)} D_t.\] 
\end{lem}
\begin{proof} We just need to prove the second isomorphism in the statement and the first isomorphism follows since the simple modules in $b_0$ are non-projective. Fix $j$. We first claim that \[(D_i\otimes D_j)_0\sim \sum_{t\in\Rect(i,j)} D_t\sim \sum_{k=0}^{m} D_{\Dl_{i,j}+2k}\sim (D_{\Dl_{i,j}}+D_{\Dl_{i,j}+2}+\cdots+D_{\Dr_{i,j}}).\] We prove by induction on $i$. The case when $i=0$ is obvious because $\Dl_{i,j}=j=\Dr_{i,j}$. Let $i\in [1, p-3]$. To complete our proof, since $S_{i+1}\sim D_{i+1}+D_i$, by Lemma \ref{lem:compSiDj}, it suffices to check that 
\begin{equation}\label{Eq:DSpecht2}
  \sum_{t=\Sl_{i+1,j}}^{\Sr_{i+1,j}} D_t\sim (D_{\Dl_{i+1,j}}+D_{\Dl_{i+1,j}+2}+\cdots+D_{\Dr_{i+1,j}})+ (D_{\Dl_{i,j}}+D_{\Dl_{i,j}+2}+\cdots+D_{\Dr_{i,j}})
\end{equation} for some $m$. For this, since $\Dl_{i+1,j}=\Dl_{i,j}\pm 1$ and $\Dr_{i+1,j}=\Dr_{i,j}\pm 1$, in view of the right-hand side of Equation \ref{Eq:DSpecht2}, we only need to verify that $\Sl_{i+1,j}=\min\{\Dl_{i,j},\Dl_{i+1,j}\}$ and $\Sr_{i+1,j}=\max\{\Dr_{i,j},\Dr_{i+1,j}\}$. However, our last assertions are clearly correct. Thus, we obtain the composition factors of $(D_i\otimes D_j)_0$ as claimed. Since $\Ext^1_{\F\sym{p}}(D_t,D_{t+2})=0$ for every $t\in [0,p-4]$, by Lemma \ref{lem:Ext1}(ii), $(D_i\otimes D_j)_0$ is semisimple. Our proof is now complete. 
\end{proof}

\begin{thm}\label{T:StableGreenRing} Let $i,j,k,\ell\in [0,p-2]$. Then $\Omega^k(D_i)\otimes \Omega^\ell(D_j)\cong \bigoplus_{t\in\Rect(i,j)}\Omega^{k+\ell}(D_t)$.
\end{thm}
\begin{proof} The isomorphism is obtained using $\Omega^0(\Omega^k(U)\otimes \Omega^\ell(V))\cong \Omega^{k+\ell}(U\otimes V)$ and Lemma \ref{L: DiDj semisimple}. 
\end{proof}

For a finite group $G$ with a cyclic Sylow $p$-subgroup, in \cite[\S5]{Janusz:1966},  Janusz constructed all the indecomposable modules for $\F G$. These modules are more generally known as the string modules for Brauer tree algebras. Applied to our case, they are the syzygies of the simple modules. We now use our Theorem \ref{T:StableGreenRing} to prove this fact. 

\begin{cor}\label{C:OmegaiDjStructure} Let $j\in [0,p-2]$ and $i\in [0,2p-3]$. We have
\begin{enumerate}[(i)]
  \item if $i\in [0,p-2]$, then $\Omega^i(D_j)$ has radical layers isomorphic with $\bigoplus_{t\in \Rect(i,j)}D_t$ and $\bigoplus_{t\in \Rect(i-1,j)}D_t$ from top to bottom,
  \item if $i\in [p-1,2p-3]$, then $\Omega^i(D_j)$ has radical layers isomorphic with $\bigoplus_{t\in \Rect(i-p+1,p-2-j)}D_t$ and $\bigoplus_{t\in \Rect(i-p,p-2-j)}D_t$ from top to bottom;
\end{enumerate} with the convention $\Rect(-1,*)=\varnothing$. In particular, all indecomposable non-projective $\F\sym{p}$-modules have period $2p-2$. Furthermore, the complete set of indecomposable $\F\sym{p}$-modules is given by 
\[\{\Omega^i(D_j):i\in [0,p-2],\ j\in[0,p-2]\}\cup \{P_j:j\in [0,p-2]\}\cup \{D^\mu:\mu\not\in \block_0\}.\]
\end{cor}
\begin{proof} Suppose first that $i\in [0,p-2]$. The module $S_i$ has a submodule $W\cong D_{i-1}$ such that $S_i/W\cong D_i$. Tensoring with $D_j$, we have \[(S_i\otimes D_j)_0/(W\otimes D_j)_0\cong (D_i\otimes D_j)_0.\] By Lemma \ref{L: DiDj semisimple}, both $(W\otimes D_j)_0$ and $(D_i\otimes D_j)_0$ are semisimple, and $\Ext^1_{\F\sym{p}}(T,T')=0$ for any two composition factors $T,T'$ of $(W\otimes D_j)_0$. In particular, $\Rad((S_i\otimes D_j)_0)\subseteq (W\otimes D_j)_0$. Furthermore, by Lemma \ref{lem:compSiDj}, $(S_i\otimes D_j)_0$ is multiplicity-free. Applying Lemma \ref{lem:Ext1} gives us $\Rad((S_i\otimes D_j)_0)=(W\otimes D_j)_0$. Using Corollary \ref{cor:0block&Heller}, part (i) follows. Suppose now that $i\in [p,2p-2]$. We have \[\Omega^i(D_j)\cong \Omega^{i-p+1}\Omega^{p-1}(D_j)\cong \Omega^{i-p+1}(D_{p-2-j}).\] Part (ii) now follows using part (i). 

When $\Omega^i(D_j)$ is simple, we have $i=0,p-1$. But $\Omega^{p-1}(D_j)\cong D_{p-2-j}\not\cong D_j$ and $\Omega^{2p-2}(D_j)\cong D_j$. It follows that the period of $D_j$ is $2p-2$. If $\Omega^i(D_j)\cong \Omega^k(D_\ell)$ for some $i,j,k,\ell\in [0,p-2]$ with $i\ge k$, then $\Omega^{i-k}(D_j)\cong D_\ell$. Thus, $i-k=0$ and $j=\ell$. By Theorem \ref{T:NumInd}, the modules in our statement give us the correct number of non-isomorphic indecomposables. Since every indecomposable non-projective $\F\sym{p}$-module is a syzygy of a simple module and each simple module has period $2p-2$, it also has period $2p-2$.
\end{proof}

\begin{eg} Let $p=5$. We have the following Auslander-Reiten quiver for $\F\sym{5}$.
\[
\begin{tikzpicture}[xscale=1.7, yscale=1.5]
\foreach \i in {1,3,5,7}
{\draw[->, shorten >=15pt, shorten <=15pt] (\i,1)--(\i+1,0);
\draw[->, shorten >=15pt, shorten <=15pt] (\i,3)--(\i+1,2);
\draw[->, shorten >=15pt, shorten <=15pt] (\i,1)--(\i+1,2);}

\foreach \i in {2,4,6}
{\draw[->, shorten >=15pt, shorten <=15pt] (\i,0)--(\i+1,1);
\draw[->, shorten >=15pt, shorten <=15pt] (\i,2)--(\i+1,1);
\draw[->, shorten >=15pt, shorten <=15pt] (\i,2)--(\i+1,3);}

\foreach \i/\j in {2/0, 6/0, 2/2, 6/2}
{\draw[->, shorten >=15pt, shorten <=20pt] (\i,\j)--(\i+1,\j);}

\foreach \i/\j in {3/0, 7/0, 3/2, 7/2}
{\draw[->, shorten >=25pt, shorten <=10pt] (\i,\j)--(\i+1,\j);}

\foreach \i in {0,8}
{\draw[->, shorten >=15pt, shorten <=20pt,dashed] (\i,2)--(\i+1,3);
\draw[->, shorten >=15pt, shorten <=20pt,dashed] (\i,2)--(\i+1,1);
\draw[->, shorten >=15pt, shorten <=20pt,dashed] (\i,0)--(\i+1,1);}

\foreach \i/\j/\label in {1/3/\Omega^6(D_0), 3/3/\Omega^4(D_0)=D_3, 5/3/\Omega^2(D_0), 7/3/D_0, 2/2/\Omega^5(D_1), 3/2/P_2, 4/2/\Omega^3(D_1), 6/2/\Omega(D_1), 7/2/P_1, 8/2/\Omega^{-1}(D_1), 1/1/\Omega^2(D_1), 3/1/D_1, 5/1/\Omega^{-2}(D_1), 7/1/\Omega^{-4}(D_1)=D_2, 2/0/\Omega(D_0), 3/0/P_0, 4/0/\Omega^{-1}(D_0), 6/0/\Omega^{-3}(D_0), 7/0/P_3, 8/0/\Omega^{-5}(D_0)
}
{\node at (\i,\j) {$\label$};}

\end{tikzpicture}
\]
\end{eg}

In \cite[Theorem 1]{Alperin/Janusz:1973}, Alperin and Janusz determined the minimal projective resolution of the trivial module for any finite group with cyclic Sylow $p$-group. In the special case of $\F\sym{p}$, the above corollary gives us the minimal projective resolution of all indecomposable non-projective $\F\sym{p}$-modules.

\begin{cor} Let $i,j\in [0,p-2]$ and $n\in\Z_{\geq 0}$. Suppose that $n+i\equiv \ell(\mod 2p-2)$ where $\ell\in [0,2p-3]$. Then we have the minimal projective resolution \[\cdots\to Q_2\to Q_1\to Q_0\to \Omega^i(D_j)\to 0\] where $Q_n$ is isomorphic to $\bigoplus_{t\in\Rect(\ell,j)}P_t$ if $\ell\in [0,p-2]$ and $\bigoplus_{t\in\Rect(\ell-p+1,p-2-j)}P_t$ if $\ell\in [p-1,2p-3]$.
\end{cor}
\begin{proof} Since $Q_n$ is the projective cover of $\Omega^{n+i}(D_j)\cong \Omega^\ell(D_j)$ and is determined by the head of $\Omega^\ell(D_j)$, the result follows from Corollary \ref{C:OmegaiDjStructure}.  
\end{proof}

Let $A=\F\sym{p}$ and $A^e=A\otimes A^{\text{op}}$ be the enveloping algebra of $A$. As another application, we draw a connection with the minimal projective resolution of the $A^e$-module $A$, which in turn relates to the Hochschild cohomology of $A$. For this, recall that the principal indecomposable $A^e$-modules are $P(\lambda,\mu)=P^\lambda\otimes P^\mu$ where $\lambda,\mu$ run through all $p$-regular partitions of $p$. In general, the tensor products of projective indecomposables of a finite-dimensional algebra over an algebraically closed field yield a complete set of projective indecomposables for the enveloping algebra (see \cite[Proposition 1.1]{kuelshammer}).

\begin{cor} Let $A=\F\sym{p}$. The minimal projective resolution of the $A^e$-module $A$ is given by \[\cdots\to R_n\to\cdots\to R_1\to R_0\to A\to 0\] where the number $r^n_{\lambda,\mu}$ of $P(\lambda,\mu)$ appearing as a direct summand of $R_n$ is given as follows.
	\begin{enumerate}[(i)]
		\item If $n=0$, $r_{\lambda,\mu}^0=\delta_{\lambda,\mu}$,
		\item Suppose that $n\ge 1$. We have $r_{\lambda,\mu}^n=0$ if $\lambda\not\in \block_0$ or $\mu\not\in \block_0$. If $\lambda=D_i$ and $\mu=D_j$ for some $i,j\in [0,p-2]$, let $s\in [1,2p-2]$ such that $n\equiv s(\mod 2p-2)$, then \[r^n_{\lambda,\mu}=\left \{\begin{array}{ll} 1&\text{if $j\in \Rect(s,i)$ and $s\in [1,p-1]$,}\\
		1&\text{if $j\in \Rect(s-p+1,p-2-j)$ and $s\in [p,2p-2]$,}\\
		0&\text{otherwise.}
		\end{array}\right .\]
		\end{enumerate} 
\end{cor}
\begin{proof} By \cite[\S1.5 Lemma]{happel}, we have 
	\begin{align*}
		r^n_{\lambda,\mu}&=\dim_\F \Ext^n_A(D^\lambda,D^\mu)=\dim_\F\Hom_A(\Omega^n(D^\lambda),D^\mu).
	\end{align*} If any of $D^\lambda$ or $D^\mu$ is projective, then $r^n_{\lambda,\mu}=0$ for $n\ge 1$. The case when $D^\lambda=D_i$ and $D^\mu=D_j$ follows from Corollary \ref{C:OmegaiDjStructure}.  
\end{proof}

\section{Stably Tensor-Semisimple Algebras}

In view of Corollary \ref{C:OmegaiDjStructure}, the stable Green ring of $\F\sym{p}$ is spanned by the syzygies of the simple modules in $\block_0$. In the previous section, we describe the multiplication of such objects. In other words, we have a full description of the ring. In particular, Lemma \ref{L: DiDj semisimple} shows that the tensor product of simple modules remains semisimple modulo projectives. This phenomenon has been observed in certain classes of bialgebras and we shall explore further with examples in this section. 

\begin{defn} A bialgebra $A$ over $\F$ is stably tensor-semisimple if, for any two simple $A$-modules $T,T'$, the tensor product $T\otimes T'$ (via the coalgebra structure of $A$) is semisimple or zero modulo projectives.
\end{defn}


We have the following examples. 

\begin{prop}\label{P:HypoA} The following bialgebras are stably tensor-semisimple.
\begin{enumerate}[(i)]
    \item Any bialgebra that is semisimple (as an algebra).
    \item The group algebra $\F G$ where $G$ is a $p$-group.
    \item The Drinfeld double of a generalized Taft algebra.
    \item Sweedler's 4-dimensional Hopf algebra. 
    \item The symmetric group algebra $\F\sym{p}$.
    \item The pointed rank one Hopf algebras of nilpotent type.
\end{enumerate} 
\end{prop}
\begin{proof} It is clear for parts (i) and (ii). For part (iii), see \cite[Proposition 2.8]{Erdmann/Green/Snashall/Taillefer}. For part (iv), see \cite[\S5.24]{Benson:2024}. For part (v), notice that the simple module $D^\mu$ is projective if $\mu$ does not belong in the principal block $\block_0$ of $\F\sym{p}$. Therefore, $D^\mu\otimes M$ is projective for any $\F\sym{p}$-module $M$ (see, for example, \cite[\S7, Lemma 4]{alp}). This part now follows using Lemma \ref{L: DiDj semisimple}. For part (vi), in \cite[Section 2]{WangLiZhang:2014}, over an algebraically closed field $K$ of characteristic zero, such algebra $H$ satisfies $\pi:H\to H/(y)\cong KG$ for a finite group $G$. Moreover, the simple $H$-modules are obtained from the simple $KG$-modules by inflation along $\pi$. Since $\Delta(y)=y\otimes g+1\otimes y$ (where $g$ is certain element in the center of $G$), for any two simple $H$-modules $V,V'$, $y$ again acts trivially on $V\otimes V'$. Therefore, $V\otimes V'$ factors through $KG$ (which is semisimple) and hence is semisimple.
\end{proof}

Following our discussion above, given a bialgebra $A$ that is stably tensor-semisimple, we can construct a $\Z$-algebra $\Upsilon(A)$ with a formal basis \[\{[T]:\text{$T$ is a simple and non-projective $A$-module}\}\] in the stable Green ring of $A$ and multiplication defined as \[[T]\cdot [T']:=\sum a_{T,T',T''}[T'']\] if $\Omega^0(T\otimes T')\cong \bigoplus a_{T,T',T''}T''$. The algebra $\Upsilon(A)$ has finite rank if $A$ is finite-dimensional. It is commutative if the coalgebra structure of $A$ is cocommutative. It is unital if $A$ has a `trivial module', i.e., one-dimensional $A$-module $\F$ such that $T\otimes \F\cong T\cong \F\otimes T$.  

\bigskip
It would be interesting to investigate the structure of $\Upsilon(A)$ and its relationship with the stable Green ring of $A$. Motivated by the tensor-product formula for $\F\sym{p}$ in Theorem \ref{T:StableGreenRing}, we compute two examples below that may shed some light on this question.

\begin{eg} We consider $\Upsilon(\F\sym{3})$ with $p=3$. It is rank two with a $\Z$-basis $\{[D_0],[D_1]\}$, and is isomorphic to $\Z C_2\cong \Z[X]/(X^2-1)$.  
\end{eg}

\begin{eg} We consider $B:=\Upsilon(\F\sym{5})$ with $p=5$. It is rank 4 with a basis $\{[D_i]:i=0,1,2,3\}$. We use the labelling $1=[D_0]$, $x=[D_1]$, $y=[D_2]$ and $z=[D_3]$. According to Lemma \ref{L: DiDj semisimple}, the multiplication table for the basis elements is given as follows:
\begin{center}
\begin{tblr}{
  colspec = {ccccc},
  cells   = {mode=math},
  hlines,
  vlines,
  hline{2} = {1.2pt},
  vline{2} = {1.2pt},
}
\cdot & 1 & x   & y   & z \\
1     & 1 & x   & y   & z \\
x     & x & 1+y & x+z & y \\
y     & y & x+z & 1+y & x \\
z     & z & y   & x   & 1
\end{tblr}
\end{center}

Let $\K$ be a field of characteristic $q$. To avoid technical difficulty, we may assume that $\K$ is algebraically closed. However, the reader will find the sufficient conditions for the field $\K$ in the calculation below. We claim that $B_\K:=\K\otimes_\Z B$ is semisimple if and only if $q\ne 2,5$. 

Suppose that $q=2$. We have a decomposition \[B_\K=eB_\K\oplus e'B_\K\] where $e=\lambda+y$ and $e'=\mu+y$ where $\lambda,\mu$ are distinct roots of the irreducible polynomial $X^2+X+1=0$. The $B_\K$-module $eB_\K$ (respectively, $e'B_\K$) is two-dimensional indecomposable with $T:=\Span\{e+ez\}$ (respectively, $T':=\Span\{e'+e'z\}$) the unique simple submodule. Furthermore, $eB_\K/T\cong T\not\cong T'\cong e'B_\K/T'$ and $eB_\K\cong KC_2\cong e'B_\K$. In particular, $\Rad(B_\K)=\Span\{e+ez,e'+e'z\}$. The Ext-quiver of $B_\K$ consists of two vertices with one loop for each vertex. 

Suppose that $q=5$. We have a decomposition \[B_\K=eB_\K\oplus e'B_\K\] where $e=\frac{1}{2}(1+z)$ and $e'=1-e$. The $B_\K$-module $eB_\K$ (respectively, $e'B_\K$) is two-dimensional indecomposable with $T:=\Span\{e+3ex\}$ (respectively, $T':=\Span\{e'+2e'x\}$) the unique simple submodule. Furthermore, $eB_\K/T\cong T\not\cong T'\cong e'B_\K/T'$. In particular, $\Rad(B_\K)=\Span\{e+3ex,e'+2e'x\}$. Similarly, the Ext-quiver of $B_\K$ consists of two vertices with one loop for each vertex. 

Suppose that $q\ne 2,5$. Again, we have a decomposition \[B_\K=eB_\K\oplus e'B_\K\] where $e=\frac{1}{2}(1+z)$ and $e'=1-e$. Let $\lambda,\mu$ be distinct roots of the polynomial $X^2-X+\frac{1}{5}=0$. Notice that $eB_\K$ has a basis $\{e,ex\}$ where $(ex)^2=e+ex$. Furthermore, the algebra $eB_\K$ decomposes as $\Span\{\lambda e+(1-2\lambda) ex\}\oplus \Span\{\mu e+(1-2\mu) ex\}$. On the other hand, let $e'=1-e$, the algebra $e'B_\K$ decomposes as $\Span\{\lambda e'+(2\lambda-1) e'x\}\oplus \Span\{\mu e'+(2\mu-1) e'x\}$. Thus, $B_\K$ is semisimple. 

In any case of $q$, we observe that $B_\K$ is basic. It remains mysterious to the authors the underlying reasons how the algebra $B_\K$ manages to single out the prime $q=5=p$ (except $q=2$) for non-semisimplicity. 
\end{eg}

\section{Benson--Symonds invariants for $\F\sym{p}$}\label{S:BensonSymonds}

The representation theory for $\F C_p$ is well-known. Its complete set of indecomposables consists of the Jordan blocks $J_i$'s of size $i$, one for each $i\in [1,p]$.  We let $E=\langle (1,2,\ldots,p)\rangle\subseteq \sym{p}$. 

In this section, we compute the Benson--Symonds invariant for all the indecomposable non-projective $\F\sym{p}$-modules. To prove our result, we make use of the following theorem. The definition of core-bounded species can be found in \cite[Definition 9.12]{bensonsymonds}.

\begin{thm}[{\cite[Theorem 10.1]{bensonsymonds}}]\label{T:BS10.1} Let $k\in [1,p-1]$. We have $\gamma_{C_p}(J_k) = \dfrac{\sin(\frac{k\pi}{p})}{\sin(\frac{\pi}{p})}$ and $\gamma_{C_p}$ agrees on module classes with a certain core-bounded species of the Green ring of $\F C_p$ and hence is in particular additive and multiplicative on $\F C_p$-modules. 
\end{thm}

\begin{lem}\label{lem:  wangslemma} \cite[Lemma 4.2.8]{wang} Let $j\in [0,p-2]$. Then 
\[\Omega^0(\Res_{E} (D_j)) \cong  \left \{
\begin{array}{ll}
J_{j+1}& \text{if $j$ is even,} \\
J_{p-j-1}&\text{if $j$ is odd.}
\end{array}\right .\]  
\end{lem}
\begin{proof} The reference we have cited here applies not just for $\sym{p}$. In our special case, the situation is much simpler. The restriction of $D_1$ to $E$ is isomorphic to $J_{p-2}$ (see \cite[Theorem 4.5]{limwang}). Then $\Res_E(D_j)\cong \bigwedge^j(\Res_E(D_1))\cong \bigwedge^j(J_{p-2})$ can be calculated using the Gaussian polynomial (see \cite[Theorem 2.8.2]{bensonpgroups}) as a module for the cyclic group of order $p$.
\end{proof}

\begin{lem}\label{L:gammaAM} Let $G$ be a finite group with a cyclic Sylow $p$-subgroup of order $p$. Then $\gamma_G$ is additive and multiplicative on $\F G$-modules, that is, $\gamma_G:\C\otimes_\Z a(\F G)\to \C$ is a $\C$-algebra homomorphism. In particular, \cite[Conjecture 1.4]{bensonsymonds} holds for $\F G$, i.e., for any $\F G$-module, we have $\gamma_G(M\otimes M^*)=\gamma_G(M)^2$. 
\end{lem}
\begin{proof} Let $E$ be a Sylow $p$-subgroup of $G$. Since it is the unique elementary abelian $p$-subgroup of $G$ up to conjugation, by Theorem \ref{T:BensonSymonds}(iii), we have $\gamma_G(M)=\gamma_E(M)$ for any $\F G$-module $M$. Therefore, for another $\F G$-module $N$, using Theorem \ref{T:BS10.1}, we have  \[\gamma_{G}(M\oplus N)=\gamma_E (M\oplus N)=\gamma_E(M)+\gamma_E(N)=\gamma_{G}(M)+\gamma_{G}(N).\] A similar argument shows that $\gamma_{G}(M\otimes N)=\gamma_{G}(M)\gamma_{G}(N)$. Clearly, $\gamma_G(F)=1$. For the final statement, using Theorem \ref{T:BensonSymonds}(iv), we have \[\gamma_G(M\otimes M^*)=\gamma_G(M)\gamma_G(M^*)=\gamma_G(M)^2.\]
\end{proof}

We have the following main result of this section.

\begin{thm}\label{T: gammaDi} Let $i \in \Z$. Then, for any $j\in [0,p-2]$, we have \[\gamma_{\sym{p}} (\Omega^i(D_j)) = \dfrac{\sin(\frac{(j+1)\pi}{p})}{\sin(\frac{\pi}{p})}.\] Furthermore, 
\begin{enumerate}[(i)]
    \item $\gamma_{\sym{p}}$ is additive and multiplicative on modules and so the above determines $\gamma_{\sym{p}}$ on any $\F\sym{p}$-module, and
    \item $\gamma_{\sym{p}}(M\otimes M^*)=\gamma_{\sym{p}}(M)^2$ for any $\F\sym{p}$-module $M$.
\end{enumerate}
\end{thm}
\begin{proof} We first consider the case when $i=0$. Clearly, $\gamma_{\sym{p}}(D_0)=1$. For $j\in [1,p-2]$, we have 
\begin{equation}\label{Eq:gamma}
\gamma_{\frak{S}_p} (D_j) = \gamma_{E}(\Res_{E}D_j)=\left \{
\begin{array}{ll}
\gamma_{E}(J_{j+1})& \text{if $j$ is even,} \\
\gamma_{E}(J_{p-j-1})&\text{if $j$ is odd,}
\end{array}\right .
\end{equation} where the first equality follows from Theorem \ref{T:BensonSymonds}(iii) because $E$ is the only elementary abelian $p$-subgroup of $\sym{p}$ up to conjugation and the second equality follows from Theorem \ref{T:BensonSymonds}(i). We now apply Theorem \ref{T:BS10.1}. For any $i\in\Z$, since $D_j$ is periodic (of period $2p-2$) by Corollary \ref{C:OmegaiDjStructure}, our result now follows from Equation \ref{Eq:gamma} and Theorem \ref{T:BensonSymonds}(ii). Parts (i) and (ii) of our statement now follow from Lemma \ref{L:gammaAM} and Corollary \ref{C:OmegaiDjStructure}.  
\end{proof}

\bibliographystyle{amsplain}
\bibliography{bib}

@article {WangLiZhang:2014,
    AUTHOR = {Wang, Z.H. and Li, L.B. and Zhang, Y.H.},
     TITLE = {Green rings of pointed rank one {H}opf algebras of nilpotent
              type},
   JOURNAL = {Algebr. Represent. Theory},
  FJOURNAL = {Algebras and Representation Theory},
    VOLUME = {17},
      YEAR = {2014},
    NUMBER = {6},
     PAGES = {1901--1924},
      ISSN = {1386-923X,1572-9079},
   MRCLASS = {16T05 (19A22)},
MRREVIEWER = {Andrei\ Marcus},
       DOI = {10.1007/s10468-014-9484-9},
       URL = {https://doi.org/10.1007/s10468-014-9484-9},
}

@article {Mullineux:1979,
    AUTHOR = {Mullineux, G.},
     TITLE = {Bijections of {$p$}-regular partitions and {$p$}-modular
              irreducibles of the symmetric groups},
   JOURNAL = {J. London Math. Soc. (2)},
  FJOURNAL = {Journal of the London Mathematical Society. Second Series},
    VOLUME = {20},
      YEAR = {1979},
    NUMBER = {1},
     PAGES = {60--66},
      ISSN = {0024-6107,1469-7750},
   MRCLASS = {20C30},
MRREVIEWER = {H.\ K.\ Farahat},
       DOI = {10.1112/jlms/s2-20.1.60},
       URL = {https://doi-org.remotexs.ntu.edu.sg/10.1112/jlms/s2-20.1.60},
}

@article {Ford/Kleshchev:1997,
    AUTHOR = {Ford, B. and Kleshchev, A. S.},
     TITLE = {A proof of the {M}ullineux conjecture},
   JOURNAL = {Math. Z.},
  FJOURNAL = {Mathematische Zeitschrift},
    VOLUME = {226},
      YEAR = {1997},
    NUMBER = {2},
     PAGES = {267--308},
      ISSN = {0025-5874,1432-1823},
   MRCLASS = {20C30 (05E10)},
MRREVIEWER = {G.\ D.\ James},
       DOI = {10.1007/PL00004340},
       URL = {https://doi-org.remotexs.ntu.edu.sg/10.1007/PL00004340},
}

@article {Dade:1966,
    AUTHOR = {Dade, E. C.},
     TITLE = {Blocks with cyclic defect groups},
   JOURNAL = {Ann. of Math. (2)},
  FJOURNAL = {Annals of Mathematics. Second Series},
    VOLUME = {84},
      YEAR = {1966},
     PAGES = {20--48},
      ISSN = {0003-486X},
   MRCLASS = {20.80},
MRREVIEWER = {C.\ W.\ Curtis},
       DOI = {10.2307/1970529},
       URL = {https://doi-org.remotexs.ntu.edu.sg/10.2307/1970529},
}

@article {Higman:1954,
    AUTHOR = {Higman, D. G.},
     TITLE = {Indecomposable representations at characteristic {$p$}},
   JOURNAL = {Duke Math. J.},
  FJOURNAL = {Duke Mathematical Journal},
    VOLUME = {21},
      YEAR = {1954},
     PAGES = {377--381},
      ISSN = {0012-7094,1547-7398},
   MRCLASS = {20.0X},
MRREVIEWER = {B.\ Eckmann},
       URL = {http://projecteuclid.org/euclid.dmj/1077465741},
}

@article {Srinivasan:1964,
    AUTHOR = {Srinivasan, B.},
     TITLE = {The modular representation ring of a cyclic {$p$}-group},
   JOURNAL = {Proc. London Math. Soc. (3)},
  FJOURNAL = {Proceedings of the London Mathematical Society. Third Series},
    VOLUME = {14},
      YEAR = {1964},
     PAGES = {677--688},
      ISSN = {0024-6115,1460-244X},
   MRCLASS = {20.80},
MRREVIEWER = {C.\ W.\ Curtis},
       DOI = {10.1112/plms/s3-14.4.677},
       URL = {https://doi-org.remotexs.ntu.edu.sg/10.1112/plms/s3-14.4.677},
}

@article {Lindsey:1974,
    AUTHOR = {Lindsey, II, J. H.},
     TITLE = {Groups with a t. i. cyclic {S}ylow subgroup},
   JOURNAL = {J. Algebra},
  FJOURNAL = {Journal of Algebra},
    VOLUME = {30},
      YEAR = {1974},
     PAGES = {181--235},
      ISSN = {0021-8693},
   MRCLASS = {20C20},
MRREVIEWER = {Bhama\ Srinivasan},
       DOI = {10.1016/0021-8693(74)90199-9},
       URL = {https://doi-org.remotexs.ntu.edu.sg/10.1016/0021-8693(74)90199-9},
}

@article {Feit:1966,
    AUTHOR = {Feit, W.},
     TITLE = {Groups with a cyclic {S}ylow subgroup},
   JOURNAL = {Nagoya Math. J.},
  FJOURNAL = {Nagoya Mathematical Journal},
    VOLUME = {27},
      YEAR = {1966},
     PAGES = {571--584},
      ISSN = {0027-7630,2152-6842},
   MRCLASS = {20.27},
MRREVIEWER = {Zvonimir\ Janko},
       URL = {http://projecteuclid.org/euclid.nmj/1118801774},
}

@article {Benson:2024,
    AUTHOR = {Benson, D. J.},
     TITLE = {Modular representation theory and commutative {B}anach
              algebras},
   JOURNAL = {Mem. Amer. Math. Soc.},
  FJOURNAL = {Memoirs of the American Mathematical Society},
    VOLUME = {298},
      YEAR = {2024},
    NUMBER = {1488},
     PAGES = {v+118},
      ISSN = {0065-9266,1947-6221},
      ISBN = {978-1-4704-7029-6; 978-1-4704-7854-4},
   MRCLASS = {20C20 (16T05 46J99)},
MRREVIEWER = {Aleksandr\ Panov},
       DOI = {10.1090/memo/1488},
       URL = {https://doi-org.remotexs.ntu.edu.sg/10.1090/memo/1488},
}

@article {Alperin/Janusz:1973,
    AUTHOR = {Alperin, J. L. and Janusz, G.},
     TITLE = {Resolutions and periodicity},
   JOURNAL = {Proc. Amer. Math. Soc.},
  FJOURNAL = {Proceedings of the American Mathematical Society},
    VOLUME = {37},
      YEAR = {1973},
     PAGES = {403--406},
      ISSN = {0002-9939,1088-6826},
   MRCLASS = {20J05},
MRREVIEWER = {Edward\ Cline},
       DOI = {10.2307/2039448},
       URL = {https://doi-org.remotexs.ntu.edu.sg/10.2307/2039448},
}

@article {Janusz:1966,
    AUTHOR = {Janusz, G. J.},
     TITLE = {Indecomposable representations of groups with a cyclic {S}ylow
              subgroup},
   JOURNAL = {Trans. Amer. Math. Soc.},
  FJOURNAL = {Transactions of the American Mathematical Society},
    VOLUME = {125},
      YEAR = {1966},
     PAGES = {288--295},
      ISSN = {0002-9947,1088-6850},
   MRCLASS = {20.80},
MRREVIEWER = {J.\ E.\ Adney},
       DOI = {10.2307/1994355},
       URL = {https://doi-org.remotexs.ntu.edu.sg/10.2307/1994355},
}

@article {Erdmann/Green/Snashall/Taillefer,
    AUTHOR = {Erdmann, K. and Green, E. L. and Snashall, N. and
              Taillefer, R.},
     TITLE = {Stable green ring of the {D}rinfeld doubles of the generalised
              {T}aft algebras (corrections and new results)},
   JOURNAL = {Algebr. Represent. Theory},
  FJOURNAL = {Algebras and Representation Theory},
    VOLUME = {22},
      YEAR = {2019},
    NUMBER = {4},
     PAGES = {757--783},
      ISSN = {1386-923X,1572-9079},
   MRCLASS = {16E30 (06B15 16G20 16G70 16T05 17B37 81R50)},
MRREVIEWER = {Eliezer\ Batista},
       DOI = {10.1007/s10468-018-9797-1},
       URL = {https://doi-org.remotexs.ntu.edu.sg/10.1007/s10468-018-9797-1},
}

@article {limwang,
    AUTHOR = {Lim, K. J. and Wang, J.},
     TITLE = {Small modules with interesting rank varieties},
   JOURNAL = {J. Algebra},
  FJOURNAL = {Journal of Algebra},
    VOLUME = {630},
      YEAR = {2023},
     PAGES = {198--224},
      ISSN = {0021-8693,1090-266X},
   MRCLASS = {20C20 (14J25 20C30)},
MRREVIEWER = {Katsuhiro\ Uno},
       DOI = {10.1016/j.jalgebra.2023.03.037},
       URL = {https://doi-org.remotexs.ntu.edu.sg/10.1016/j.jalgebra.2023.03.037},
}

@article {kuelshammer,
    AUTHOR = {K\"ulshammer, B.},
     TITLE = {Some indecomposable modules and their vertices},
   JOURNAL = {J. Pure Appl. Algebra},
  FJOURNAL = {Journal of Pure and Applied Algebra},
    VOLUME = {86},
      YEAR = {1993},
    NUMBER = {1},
     PAGES = {65--73},
      ISSN = {0022-4049,1873-1376},
   MRCLASS = {20C05},
MRREVIEWER = {A.\ S.\ Kondrat\cprime ev},
       DOI = {10.1016/0022-4049(93)90153-K},
       URL = {https://doi-org.remotexs.ntu.edu.sg/10.1016/0022-4049(93)90153-K},
}

@article {Brauer:1947,
	AUTHOR = {R. Brauer},
	TITLE = {On a conjecture by {N}akayama},
	JOURNAL = {Trans. Roy. Soc. Canada Sect. III},
	FJOURNAL = {Transactions of the Royal Society of Canada. Section III},
	VOLUME = {41},
	YEAR = {1947},
	PAGES = {11--19},
	ISSN = {0035-9122},
	MRCLASS = {20.0X},
	MRREVIEWER = {R.\ M.\ Thrall},
}

@article {robinsonproof,
	AUTHOR = {Robinson, G. de B.},
	TITLE = {On a conjecture by {N}akayama},
	JOURNAL = {Trans. Roy. Soc. Canada Sect. III},
	FJOURNAL = {Transactions of the Royal Society of Canada. Section III},
	VOLUME = {41},
	YEAR = {1947},
	PAGES = {20--25},
	ISSN = {0035-9122},
	MRCLASS = {20.0X},
	MRREVIEWER = {R.\ M.\ Thrall},
}

@book {alp,
    AUTHOR = {Alperin, J. L.},
     TITLE = {Local {R}epresentation {T}heory},
    SERIES = {Cambridge Studies in Advanced Mathematics},
    VOLUME = {11},
      NOTE = {Modular {R}epresentations as an {I}ntroduction to the {L}ocal
              {R}epresentation {T}heory of {F}inite {G}roups},
 PUBLISHER = {Cambridge University Press, Cambridge},
      YEAR = {1986},
     PAGES = {x+178},
      ISBN = {0-521-30660-4},
   MRCLASS = {20-02 (20C20)},
MRREVIEWER = {Peter\ W.\ Donovan},
       DOI = {10.1017/CBO9780511623592},
       URL = {https://doi-org.remotexs.ntu.edu.sg/10.1017/CBO9780511623592},
}

@book {ben1,
    AUTHOR = {Benson, D. J.},
     TITLE = {Representations and {C}ohomology. {I}},
    SERIES = {Cambridge Studies in Advanced Mathematics},
    VOLUME = {30},
   EDITION = {Second},
      NOTE = {Basic representation theory of finite groups and associative
              algebras},
 PUBLISHER = {Cambridge University Press, Cambridge},
      YEAR = {1998},
     PAGES = {xii+246},
      ISBN = {0-521-63653-1},
   MRCLASS = {20-02 (20Cxx 20J06)},
}

@phdthesis {wang,
    author = {Wang, J.},
    title = {On the {R}ank {V}arieties and {J}ordan {T}ypes of a {C}lass of {S}imple {M}odules},
    school = {Nanyang Technological University, School of Physical and Mathematical Sciences, Singapore},
    year = {2024},
}

@article {limtan,
    AUTHOR = {K. J. Lim and K. M. Tan},
     TITLE = {Periodic {L}ie modules},
   JOURNAL = {J. Algebra},
  FJOURNAL = {Journal of Algebra},
    VOLUME = {445},
      YEAR = {2016},
     PAGES = {280--294},
      ISSN = {0021-8693,1090-266X},
   MRCLASS = {20C30},
MRREVIEWER = {Craig\ J.\ Dodge},
       DOI = {10.1016/j.jalgebra.2015.08.017},
       URL = {https://doi-org.remotexs.ntu.edu.sg/10.1016/j.jalgebra.2015.08.017},
}

@book {bensonpgroups,
	AUTHOR = {Benson, D. J.},
	TITLE = {Representations of {E}lementary {A}belian {$p$}-{G}roups and {V}ector
	{B}undles},
	SERIES = {Cambridge Tracts in Mathematics},
	VOLUME = {208},
	PUBLISHER = {Cambridge University Press, Cambridge},
	YEAR = {2017},
	PAGES = {xvii+328},
	ISBN = {978-1-107-17417-7},
	MRCLASS = {20C20 (14J60 14L30)},
	MRREVIEWER = {Alan\ Koch},
	DOI = {10.1017/9781316795699},
	URL = {https://doi-org.remotexs.ntu.edu.sg/10.1017/9781316795699},
}

@article {bensonsymonds,
	AUTHOR = {D. J. Benson and P. Symonds},
	TITLE = {The non-projective part of the tensor powers of a module},
	JOURNAL = {J. Lond. Math. Soc. (2)},
	FJOURNAL = {Journal of the London Mathematical Society. Second Series},
	VOLUME = {101},
	YEAR = {2020},
	NUMBER = {2},
	PAGES = {828--856},
	ISSN = {0024-6107,1469-7750},
	MRCLASS = {20C20 (46J05)},
	MRREVIEWER = {Christopher\ P.\ Bendel},
	DOI = {10.1112/jlms.12288},
	URL = {https://doi-org.remotexs.ntu.edu.sg/10.1112/jlms.12288},
}

@book {jk,
	AUTHOR = {G. D. James and A. Kerber},
	TITLE = {The {R}epresentation {T}heory of the {S}ymmetric {G}roup},
	SERIES = {Encyclopedia of Mathematics and its Applications},
	VOLUME = {16},
	NOTE = {With a foreword by P. M. Cohn,
	With an introduction by Gilbert de B. Robinson},
	PUBLISHER = {Addison-Wesley Publishing Co., Reading, MA},
	YEAR = {1981},
	PAGES = {xxviii+510},
	ISBN = {0-201-13515-9},
	MRCLASS = {20-02 (20C30)},
	MRREVIEWER = {A.\ O.\ Morris},
}

@book {james,
	AUTHOR = {G. D. James},
	TITLE = {The {R}epresentation {T}heory of the {S}ymmetric {G}roups},
	SERIES = {Lecture Notes in Mathematics},
	VOLUME = {682},
	PUBLISHER = {Springer, Berlin},
	YEAR = {1978},
	PAGES = {v+156},
	ISBN = {3-540-08948-9},
	MRCLASS = {20C30 (20-02)},
	MRREVIEWER = {Dragomir\ \v Z.\ Djokovi\'c},
}

@article {glow,
	AUTHOR = {Giannelli, E. and Lim, K. J. and O'Donovan, W. and
	Wildon, M.},
	TITLE = {On signed {Y}oung permutation modules and signed
	{$p$}-{K}ostka numbers},
	JOURNAL = {J. Group Theory},
	FJOURNAL = {Journal of Group Theory},
	VOLUME = {20},
	YEAR = {2017},
	NUMBER = {4},
	PAGES = {637--679},
	ISSN = {1433-5883,1435-4446},
	MRCLASS = {20C30 (20C20)},
	MRREVIEWER = {Nadia\ P.\ Mazza},
	DOI = {10.1515/jgth-2017-0007},
	URL = {https://doi-org.remotexs.ntu.edu.sg/10.1515/jgth-2017-0007},
}

@incollection {happel,
	AUTHOR = {Happel, D.},
	TITLE = {Hochschild cohomology of finite-dimensional algebras},
	BOOKTITLE = {S\'eminaire d'{A}lg\`ebre {P}aul {D}ubreil et {M}arie-{P}aul
	{M}alliavin, 39\`eme {A}nn\'ee ({P}aris, 1987/1988)},
	SERIES = {Lecture Notes in Math.},
	VOLUME = {1404},
	PAGES = {108--126},
	PUBLISHER = {Springer, Berlin},
	YEAR = {1989},
	ISBN = {3-540-51812-6},
	MRCLASS = {16E40 (16E10 16G20 16P10 16S80 16W50)},
	MRREVIEWER = {Claude\ Cibils},
	DOI = {10.1007/BFb0084073},
	URL = {https://doi-org.remotexs.ntu.edu.sg/10.1007/BFb0084073},
}
\end{document}